\documentclass{amsart}

\usepackage{graphicx}
\usepackage{here}

\newtheorem{theo}{Theorem}

\theoremstyle{definition}
\newtheorem{defn}{Definition}

\theoremstyle{remark}
\newtheorem{remark}{Remark}

\begin{document}

  \title{Forbidden detour number on virtual knot}

  \author{Shun Yoshiike}
  \address{Nihon University Buzan Junior $\&$ Senior High School.  5-40-10 Otsuka, Bunkyo-ku, Tokyo 112-0012, Japan}

  \author{Kazuhiro Ichihara}
  \address{College of Humanities and Sciences, Nihon University. 3-25-40 Sakurajosui, Setagaya-ku, Tokyo 156-8550, Japan}



\begin{abstract}
We show that the forbidden detour move, essentially introduced by Kanenobu and Nelson, is an unknotting operation for virtual knots. 
Then we define the forbidden detour number of a virtual knot to be the minimal number of forbidden detour moves necessary to transform a diagram of the virtual knot into the trivial knot diagram. 
Some upper and lower bounds on the forbidden detour number are given in terms of the minimal number of real crossings or the coefficients of the affine index polynomial of the virtual knot. 
\end{abstract}

\maketitle

\section{Introduction}
As a generalization of (classical) knots in the $3$-space, Kauffman introduced \textit{virtual knots} in \cite{KF}. 
Since then various studies have been made.
For example, relations of virtual knots and Gauss diagrams were studied by Goussarov, Polyak, and Viro in \cite{GPV}. 
In their research, a kind of local move on virtual knots was introduced, which they call the \textit{forbidden move}. 
Then it was shown by Kanenobu \cite{K} and Nelson \cite{Nelson} independently that for any diagram $D$ of a virtual knot,  there exists a finite sequence of Reidemeister moves, virtual Reidemeister moves and forbidden moves that takes $D$ to the trivial knot diagram, i.e., the forbidden move is an unknotting operation for virtual knots. 

In the studies of forbidden moves in \cite{K}, Kanenobu introduced and used several moves for virtual knot diagrams. 
The two of them, called $F_2$-move and $F'_{2}$-move, which are essentially equivalent, played a key role in his arguments. 
Actually, they were also considered and used by Nelson in \cite{Nelson}. 
Later, the $F'_{2}$-move is treated by Crans, Mellor, and Ganzell in \cite{CMG}, which they called the \textit{forbidden detour move}. 
See Figure~\ref{FandFd}. 

\begin{figure}[H]
  {\unitlength=1cm
  \begin{picture}(10,2.5)
\put(1,.5){\includegraphics[scale=.11] {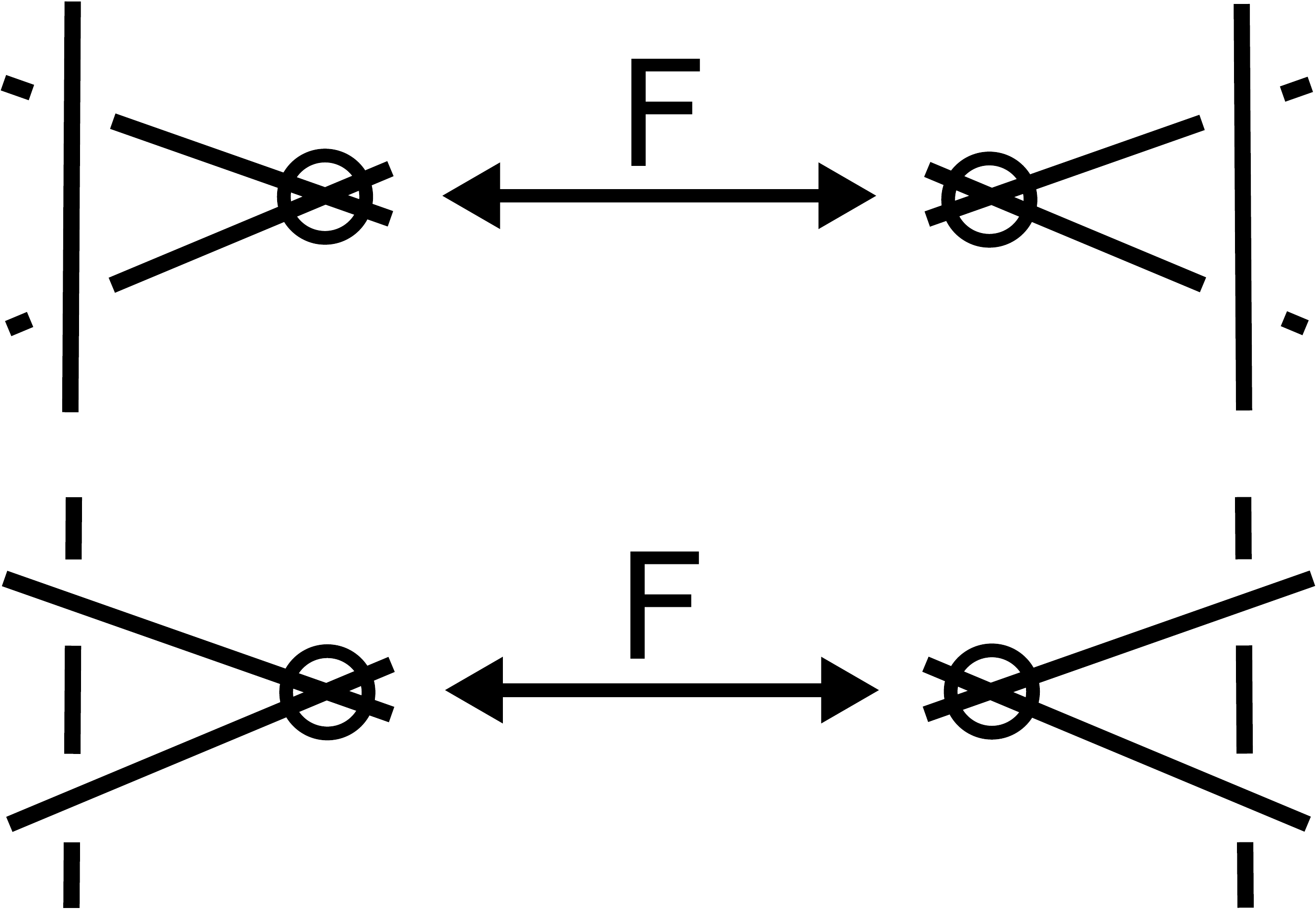} }
\put(5.1,.7){\includegraphics[scale=.14] {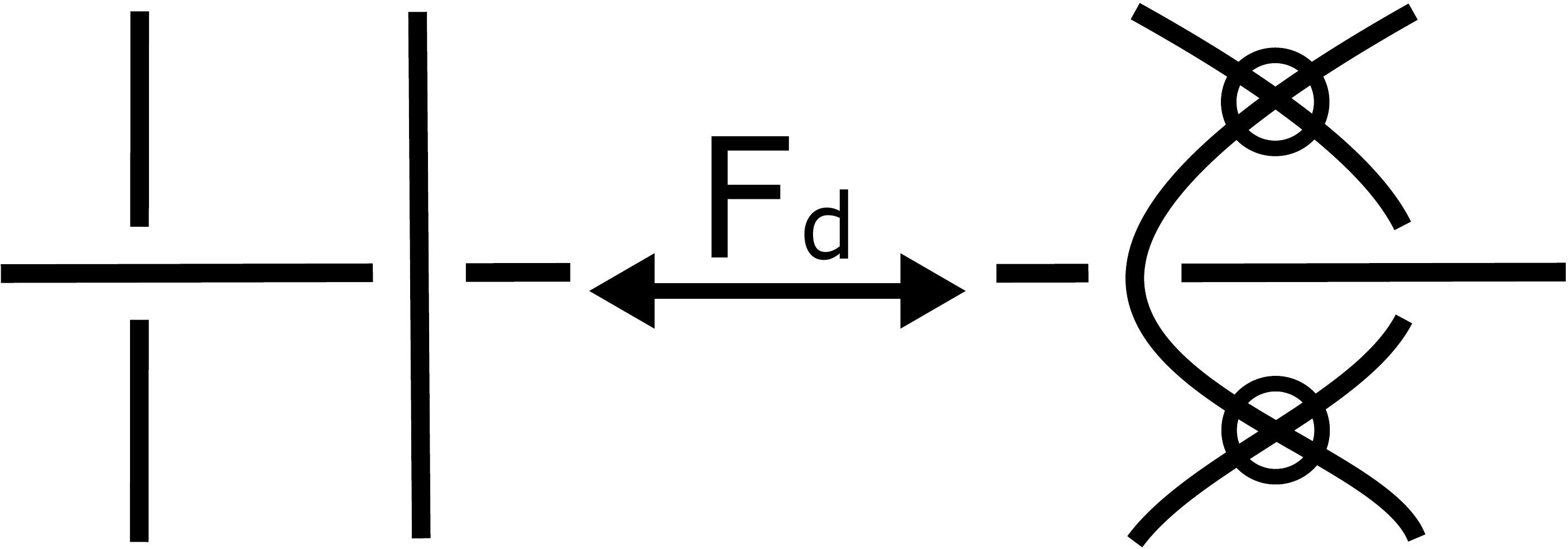} }
\end{picture}}
  \caption{Forbidden moves $F$ and forbidden detour move $F_d$}
  \label{FandFd}
\end{figure}

We here pick up that move, and obtain the following. 

\begin{theo}\label{thm1}
Let $D$ be a virtual knot diagram of a virtual knot. 
Then, $D$ can be transformed to the trivial knot diagram by using Reidemeister moves, virtual Reidemeister moves, and forbidden detour moves. 
Moreover, if $D$ has $c$ real crossings, 
then the number of forbidden detour moves is at most 
$(c-1)(2c^{2}+11c-3)/24$ if $c$ is odd and 
$c(2c^{2}+9c-14)/24$ if $c$ is even.
\end{theo}

\begin{remark}
We note that the $F_2$-move in \cite{K} (depicted in Figure~\ref{F2}), which is equivalent to the forbidden detour move, 
can be regarded as a variation of the delta move on (classical) knots, which was introduced by Matveev in \cite{Matveev} and by Murakami and Nakanishi in \cite{MurakamiNakanishi}, independently. 
They showed that the delta move is an unknotting operation for classical knots, but it is known that it is not an unknotting operation for virtual knots. 
See \cite[Theorem 1.6]{SatohTaniguchi} for example. 
\end{remark}

\begin{figure}[H]
  {\unitlength=1cm
  \begin{picture}(10,2)
\put(2.8,.5){\includegraphics[scale=.08] {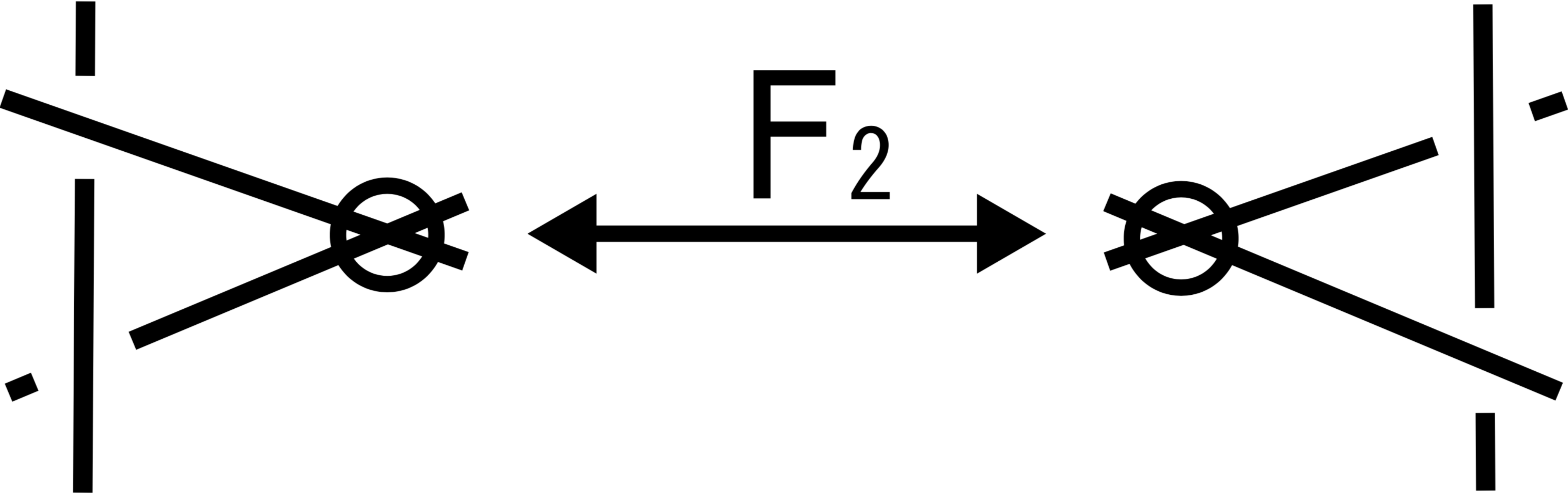} }
\end{picture}}
  \caption{$F_2$ move}
  \label{F2}
\end{figure}


In virtue of this result, we introduce the following notion. 

\begin{defn}
Let $K$ be a virtual knot.
The \textit{forbidden detour number} $F_{d}(K)$ of $K$ is defined as the minimal number of forbidden detour moves necessary to transform a virtual knot diagram of $K$ into the trivial knot diagram.
\end{defn}

We next consider  lower bounds on the forbidden detour numbers of virtual knots.
To obtain lower bounds, a variation of an invariant, called the affine index polynomial, under a forbidden detour move, plays a key role. In fact, we have the following.

\begin{theo}\label{LBFd}
Let $K$ be a virtual knot, and $P_K$ denote the affine index polynomial of $K$. Suppose that $P_{K}$ is expressed as $(t-1) \sum_{n \in \mathbb{Z}} a_{n}t^{n}$. Then, we have the following.
\[
F_{d}(K) \geq \frac{\sum_{n \in \mathbb{Z}}| a_{n}|}{2}
\]
\end{theo}

In the following, our terminology about virtual knot and Gauss diagram etc follows from those in  \cite{CMG}. 


\section{Forbidden detour number}

A \textit{virtual knot} is defined as an equivalent class of virtual knot diagrams under the Reidemeister moves, virtual Reidemeister moves. 
Also, virtual knots correspond bijectively to the equivalent classes of Gauss diagrams under moves reinterpreted Reidemeister moves. 
That is, (classical) Reidemeister moves can modify the virtual knot diagrams, but do not change the virtual knot represented by the diagrams. 

On the other hand, the forbidden move and forbidden detour move can change the virtual knots by modifying Gauss diagrams. 
In fact, as claimed in \cite[Section 2]{CMG}, the forbidden detour move gives the effect on Gauss diagrams of switching the head of one arrow with the tail of an adjacent arrow. 
See Figure~\ref{FdonGD}. 

\begin{figure}[H]
  {\unitlength=1cm
  \begin{picture}(10,2.5)
\put(2,0.5){\includegraphics[scale=.23] {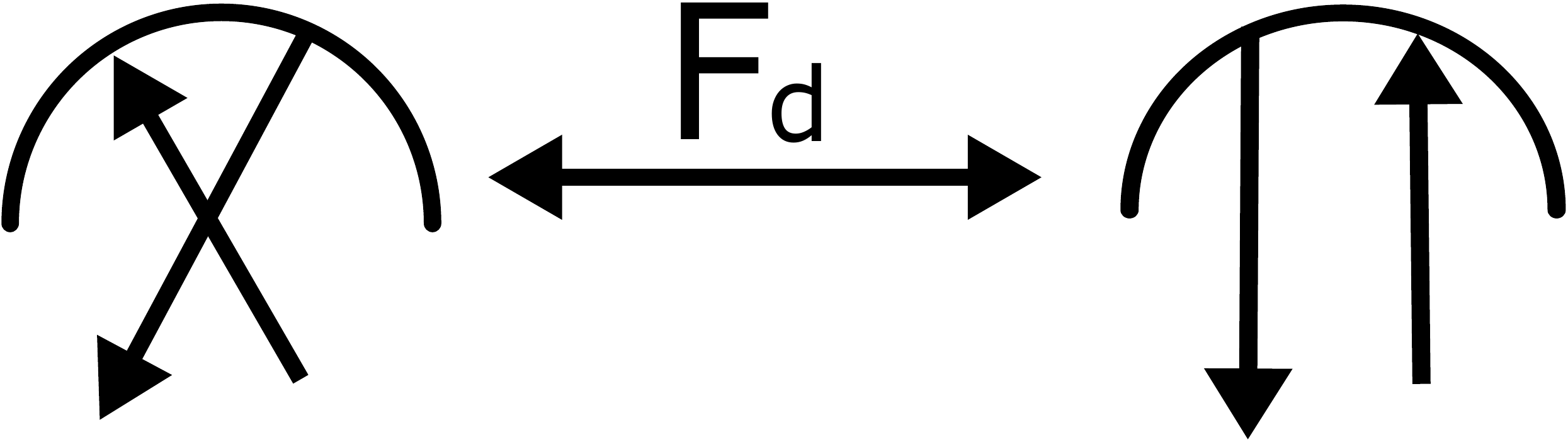} }
\end{picture}}
  \caption{The effect of an $F_{d}$-move on Gauss diagrams}
  \label{FdonGD}
\end{figure}

In the following, we call the move on Gauss diagrams corresponding a forbidden detour move also a a forbidden detour move on the Gauss diagrams. 


\begin{proof}[Proof of Theorem~\ref{thm1}]
Let $D$ be a virtual knot diagram $D$ with $c$ real crossings of a virtual knot, and $G$ a Gauss diagram associated to $D$. 
We consider an arrow $A$ of $G$, and assume that $a$ arrow-heads and $b$ arrow-tails exist on one side of the external circle of $G$ divided by the end points of $A$. 
We can suppose that and $a+b$ is smaller than or equal to $c-1$.   

\begin{figure}[H]
  {\unitlength=1cm
  \begin{picture}(10,3)
\put(4,0.5){\includegraphics[scale=.18] {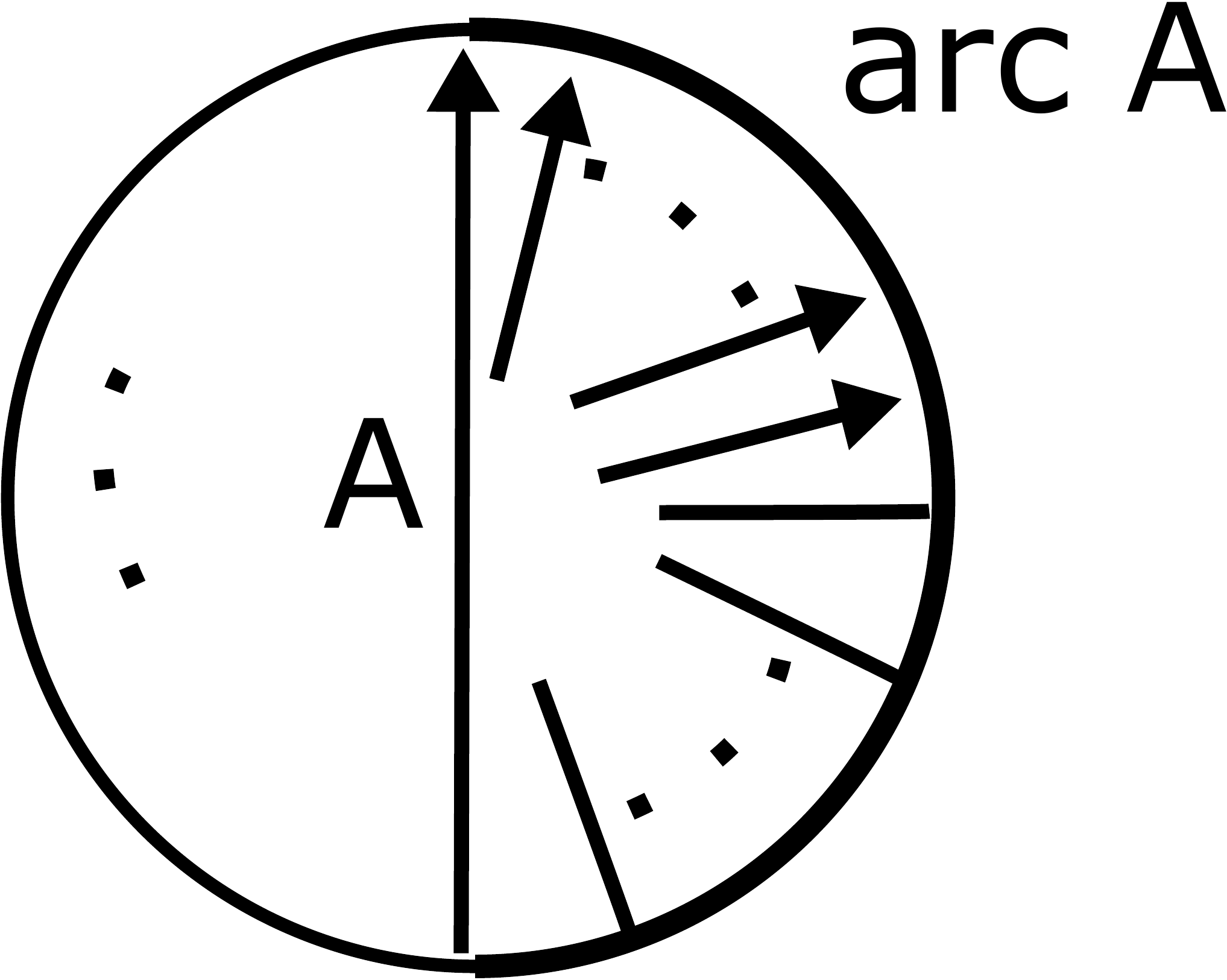} }
\end{picture}}
  \caption{Gauss diagram $G$ and arrow $A$}
  \label{GandA}
\end{figure}

Let us remove $A$ by using forbidden detour moves. 
First, we focus on all arrow-heads sandwiched between the end points of $A$ as shown Figure~\ref{SeqMoves}. 
We use forbidden detour moves at most $b+1$ times to sweep an arrow-head to outside of the considered part of the circle. We repeat this procedure $a$ times until no arrow-heads exist on that part. 
Second, we remove $A$ by using an $R_{1}$-move and forbidden detour moves $b$ times. 
Therefore, the number of forbidden detour moves to remove $A$ is at most $a(b+1)+b = a+b+ab$. 
\begin{figure}[H]
  {\unitlength=1cm
  \begin{picture}(10,7)
\put(-0.5,4.5){\includegraphics[scale=.18] {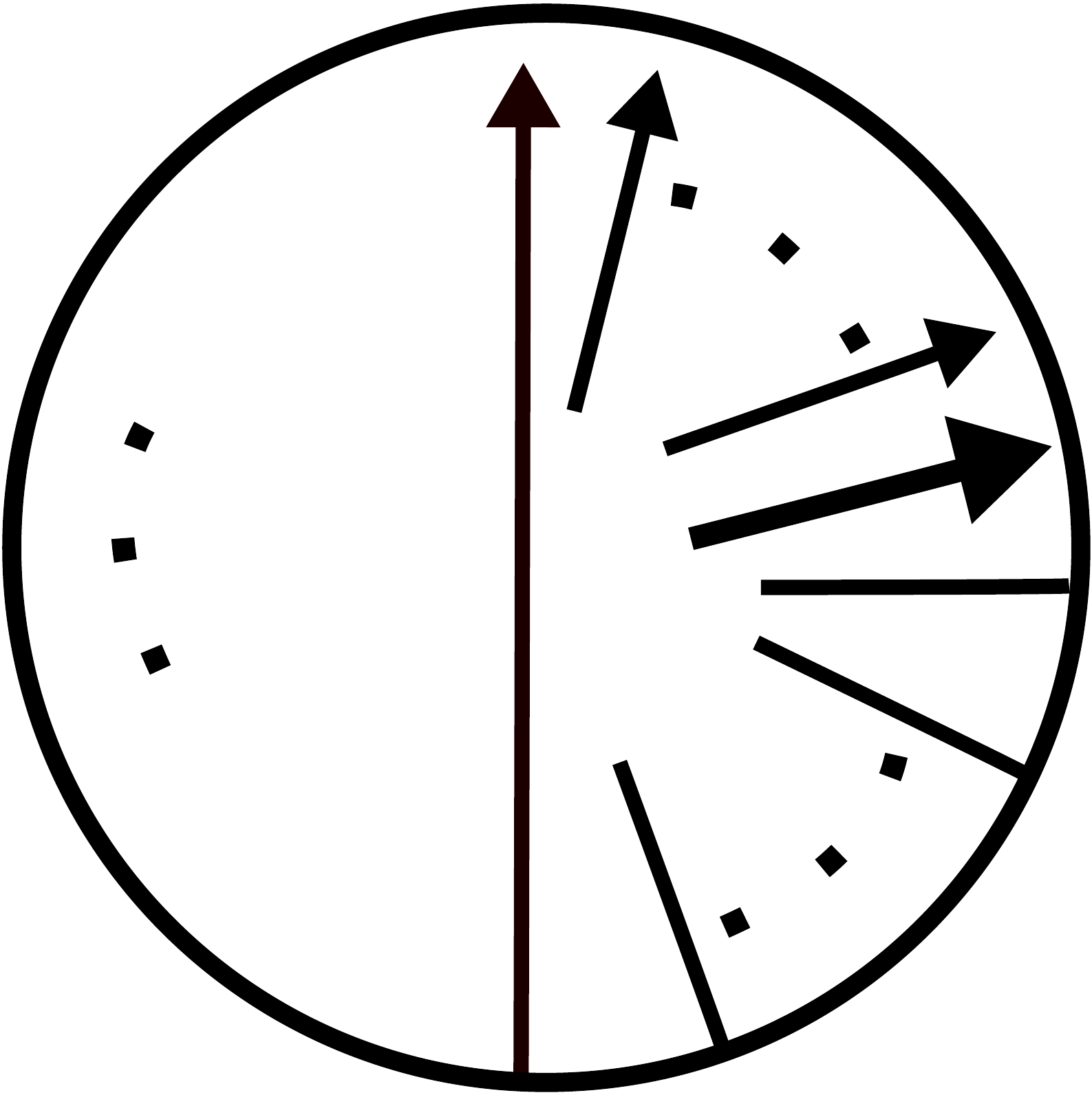} }
\put(2.7,5.5){\includegraphics[scale=.12] {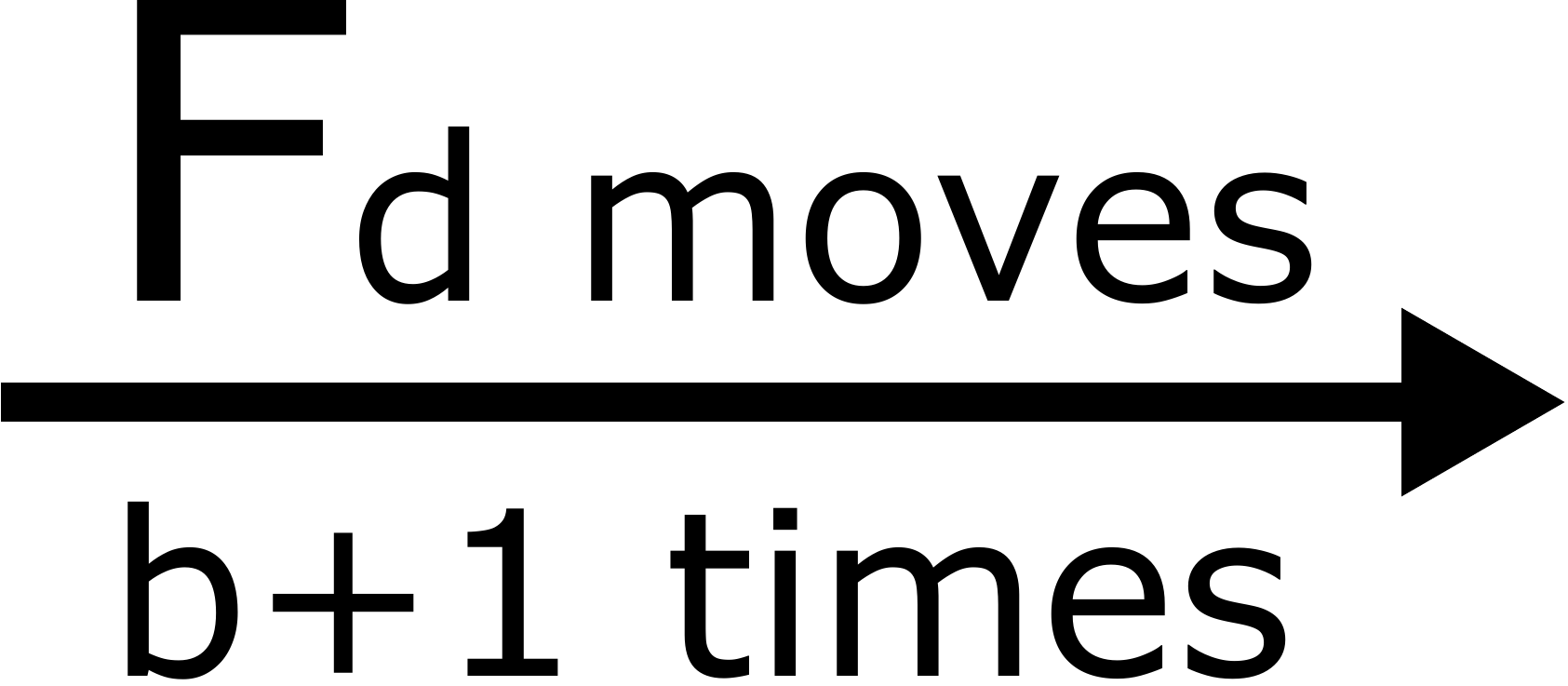} }
\put(5,4.5){\includegraphics[scale=.18] {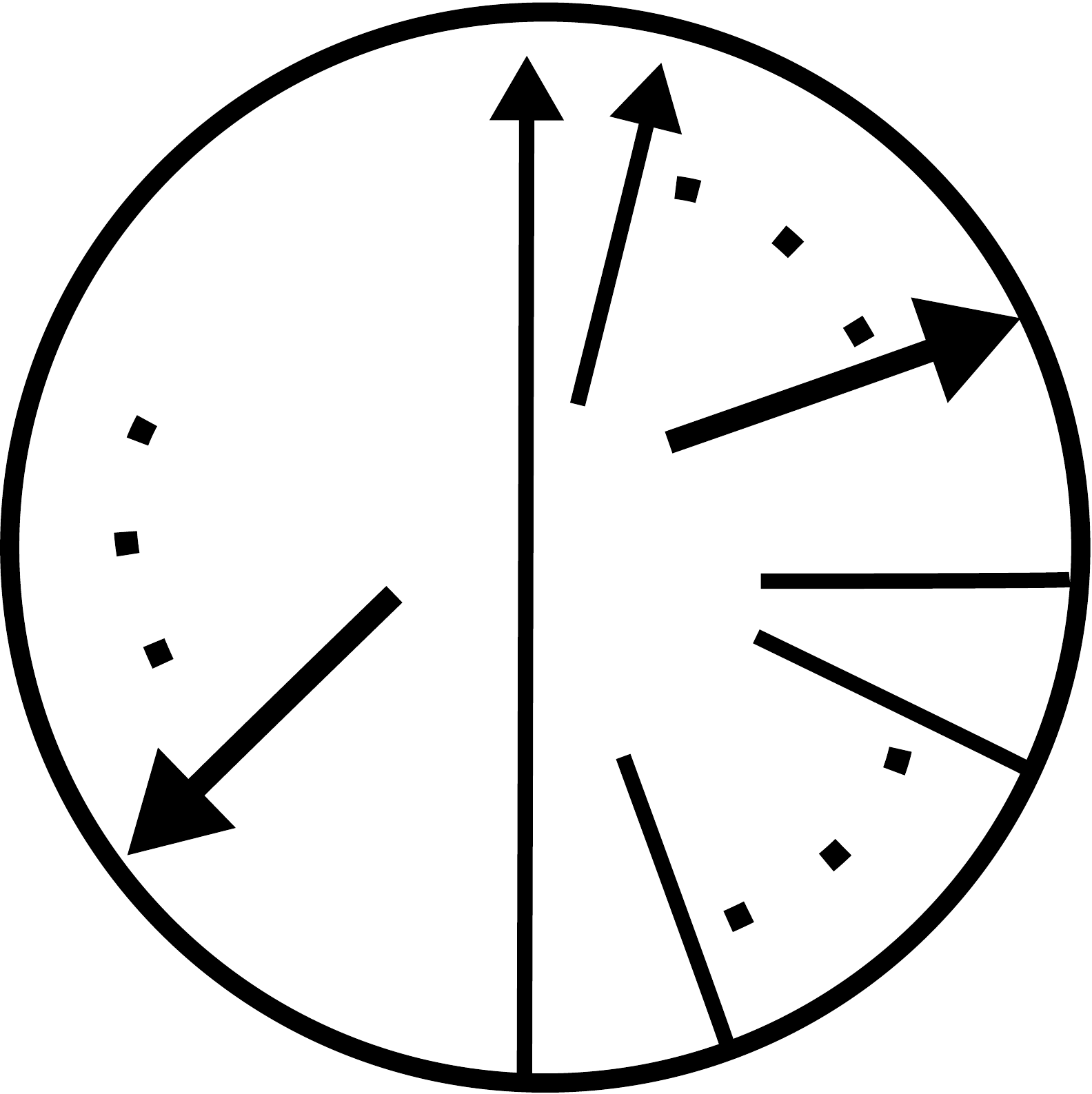} }
\put(8.2,5.5){\includegraphics[scale=.12] {fdsarrow.pdf} }
\put(-0.5,1){\includegraphics[scale=.18] {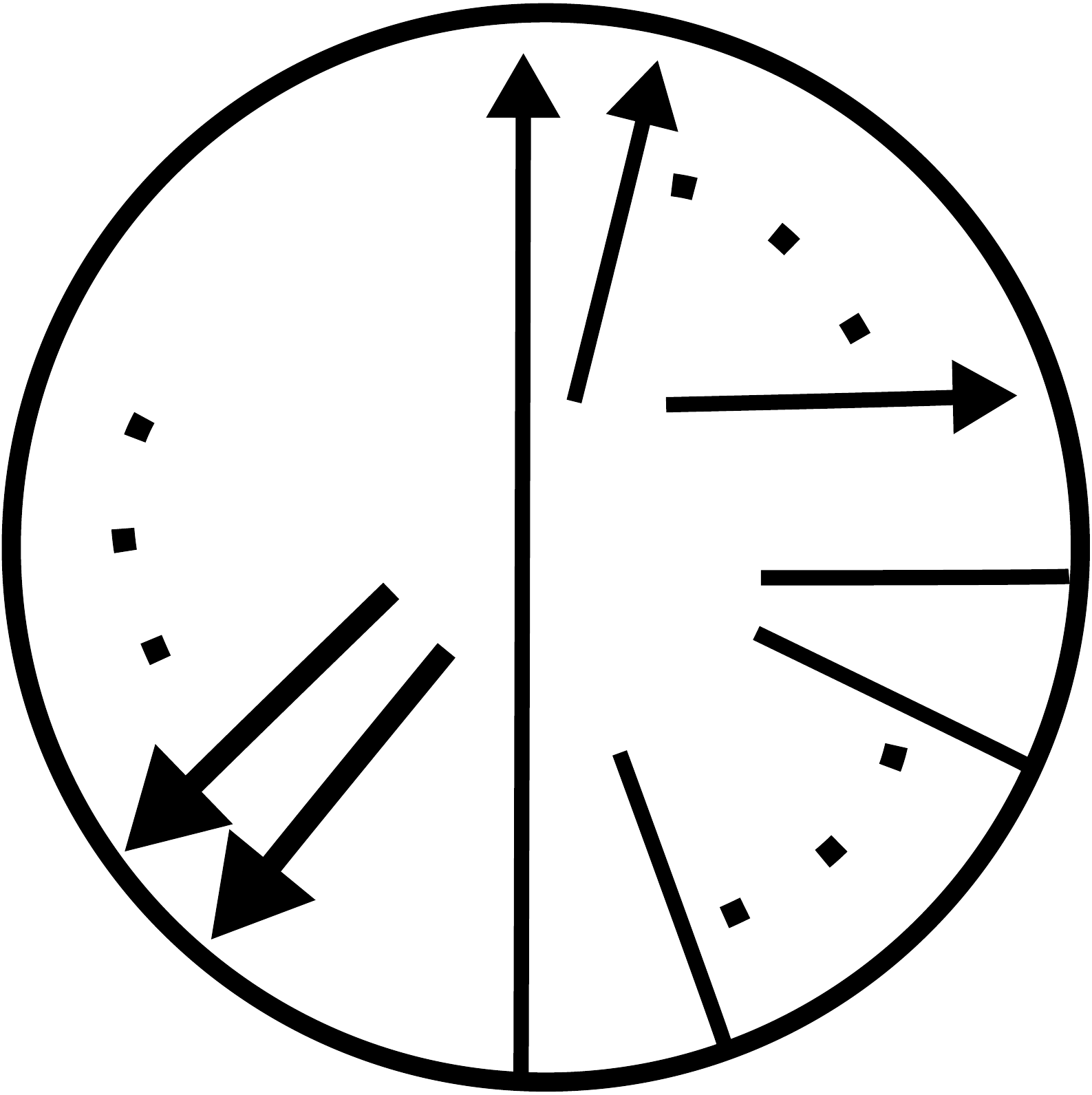} }
\put(2.6,2.3){\includegraphics[scale=.12] {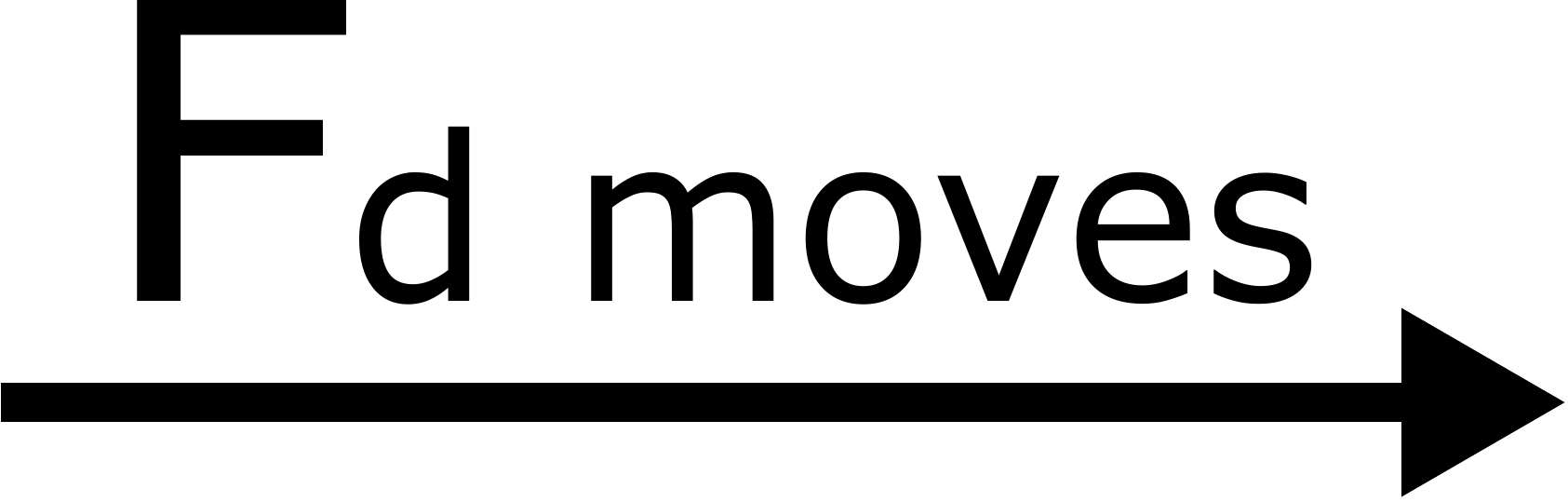} }
\put(5,1){\includegraphics[scale=.18] {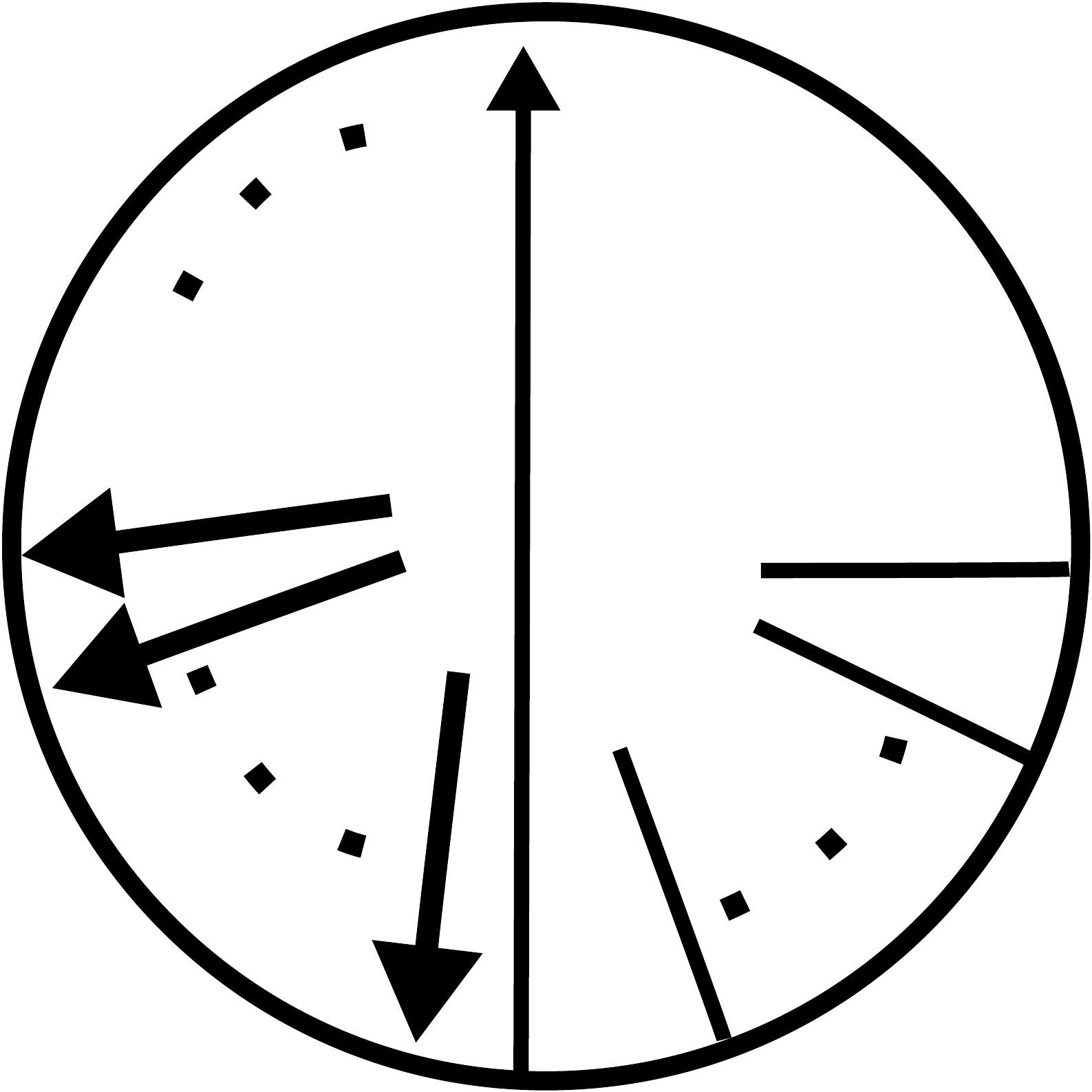} }
\end{picture}}
  \caption{Sequences of forbidden detour moves ($F_d$-moves)}
  \label{SeqMoves}
\end{figure}
Then, since $a>0$ and $b>0$, we see that $2\sqrt{ab} \leq a+b \leq c-1$, and so $4ab \leq (c-1)^{2}$. 
Then, we get
\begin{align*}
a+b+ab \leq (c-1) +\frac{(c-1)^{2}}{4} . 
\end{align*}
When $a$ is equal to $b$, the equality holds.
Let $a_{c} = \lfloor \frac{(c-1)^{2}}{4} \rfloor$. 
When $n = 2\ell$ with some $\ell \in \mathbb{N}$, we have the following.
\begin{align*}
\sum^{n}_{c=1}a_{c} &= \sum^{2\ell}_{c=1}a_{c}\\
&=\sum^{\ell}_{s=1} \{ a_{2s-1} + a_{2s} \} \\
&=\sum^{\ell}_{s=1} \{ \lfloor \frac{(2s-2)^{2}}{4} \rfloor +\lfloor \frac{(2s-1)^{2}}{4} \rfloor \} \\
&=\sum^{\ell}_{s=1}(2s^2 -3s +1 ) \\
&= \frac{1}{6}\ell(4\ell +1)(\ell -1)
\end{align*}
Then, since $\ell = n/2$, 
we get the following.
$$
\sum^{n}_{c=1}\{ \lfloor \frac{(c-1)^{2}}{4} \rfloor + c-1 \} 
=\frac{1}{24}n(2n^2 +9n -14)
$$
On the other hand, when $n =2\ell -1$ with some $\ell \in \mathbb{N}$, we have the following.
\begin{align*}
\sum^{n}_{c=1}a_{c} &= \sum^{2\ell-1}_{c=1}a_{c}\\
&=\sum^{2\ell}_{c=1}a_{c} - a_{2\ell} \\
&= \frac{1}{6}\ell(4\ell +1)(\ell -1) -(\ell^2 -\ell)\\
&= \frac{1}{6}\ell(\ell -1)(4\ell -5)
\end{align*}
Then, since $\ell = (n+1)/2$, 
we have
the following.
\begin{align*}
&\sum^{n}_{c=1}\{ \lfloor \frac{(c-1)^{2}}{4} \rfloor + c-1 \}\\
&=\frac{1}{24}(n+1)(n-1)(2n-3)+\frac{1}{2}n(n+1)-n\\
&=\frac{1}{24}(n-1)(2n^2 +11n -3)
\end{align*}

Consequently, $D$ can be transformed to the trivial knot diagram 
by using Reidemeister moves, virtual Reidemeister moves, forbidden detour moves, and 
if $D$ has $c$ real crossings, 
the number of necessary forbidden detour moves is at most 
$(c-1)(2c^{2}+11c-3)/24$ if $c$ is odd and $c(2c^{2}+9c-14)/24$ if $c$ is even.
\end{proof}

%

%


\section{Lower bound of forbidden detour number}

Next, we consider the lower bound of forbidden detour number of a virtual knot. 
In this section, we estimate it using by using an invariant, called the affine index polynomial. 
In fact, for the forbidden move, Sakurai showed in \cite{SA} the following; 
Let $K$ and $K'$ be two virtual knots which can be transformed into each other by a single forbidden move. 
Then 
\[
P_{K}-P_{K'} = (t-1)(\pm t^{ \ell }\pm t^{m})
\]
holds for some integer $\ell$ and \textit{m}, where $P_K$ denotes affine index polynomial. 
By imitating the argument in \cite{SA}, we have the following.
\begin{theo}\label{AIP}
Let $K$ and $K'$ be two virtual knots which can be transformed into each other by a single forbidden detour move. Then we have
\[
P_{K}-P_{K'} = (t-1)(\pm t^{ \ell }\mp t^{m})
\]
for some integer $\ell$ and \textit{m}, where $P_K$ denotes affine index polynomial.
\end{theo}






To prove this, we recall some definitions about the affine index polynomial used in \cite{SA}.

First, we define virtual knot invariants by indexes of arrows for a Gauss diagram. 
Let $G$ be a Gauss diagram of a virtual knot $K$, and $\gamma = \overrightarrow{PQ}$ an arrow oriented from $P$ to $Q$ with sign $\varepsilon (\gamma)$ in $G$. 
We give the signs to the endpoints $P$ and $Q$, denoted by $\varepsilon (P)$ and $\varepsilon (Q)$, respectively, such that $\varepsilon (P) = -\varepsilon (\gamma)$ and $\varepsilon (Q) = \varepsilon (\gamma)$. 
For an arrow $\gamma = \overrightarrow{PQ}$ in a Gauss diagram $G$, the \textit{specified arc} of $\gamma$ is the arc $\alpha$ in the outer circle $\mathbb{S}^1$ with endpoints $P$ and $Q$ oriented from $P$ to $Q$ with respect to the orientation of $\mathbb{S}^1$. 
The \textit{index} of $\gamma$ is the sum of the signs of all the endpoints of arrows on $\alpha$ other than $P$ and $Q$, and denoted by $i(\gamma)$. 

Then the \textit{n-writhe} $J_{n}(K)$ of of a virtual knot $K$ is defined as 
\[
J_{n}(K) = \sum_{i(\gamma) =n} \varepsilon (\gamma) \quad(n \neq 0)
\]
and, we define the \textit{affine index polynomial} $P_K$ of \textit{K} as 
\[
P_{K} = \sum_{n \in \mathbb{Z}} J_{n}(K)(t^{-n}-1) \;.
\]

We remark that this is different from the original definition by Kauffman in \cite{KFM}. 
However Sakurai showed in \cite[Proposition $3.2$]{SA} that this gives an alternative definition of the affine index polynomial.

\begin{proof}[Proof of Theorem~\ref{AIP}]
Suppose that virtual knots $K$ and $K'$ are represented by Gauss diagrams \textit{G} and $G'$ respectively. 
There are two cases, Case (I) and (I\hspace{-1pt}I) of Figure~\ref{ind}, for the change of $G$ and $G'$. 
We here only consider  the case (I) since the other case can be treated similarly. 

Let $\gamma_{i}$ be arrows $\gamma_{i}'$ ($i = 1$, $2$) of \textit{G} and $G'$ are the two arrows in the part where a forbidden detour move is applied. For arrows $\gamma_{i}$ and $\gamma_{i}'$, by Figure~\ref{ind}, we have
\begin{align*}
&i(\gamma_{1}') = i(\gamma_{1}) + \varepsilon(\gamma_{2}) \\
&i(\gamma_{2}') = i(\gamma_{2}) - \varepsilon(\gamma_{1}) 
\end{align*}
where $\varepsilon(\gamma_{i}) =\varepsilon(\gamma_{i}') $. 
\begin{figure}[H]
  {\unitlength=1cm
  \begin{picture}(10,4.5)
\put(-2,1){\includegraphics[scale=.12] {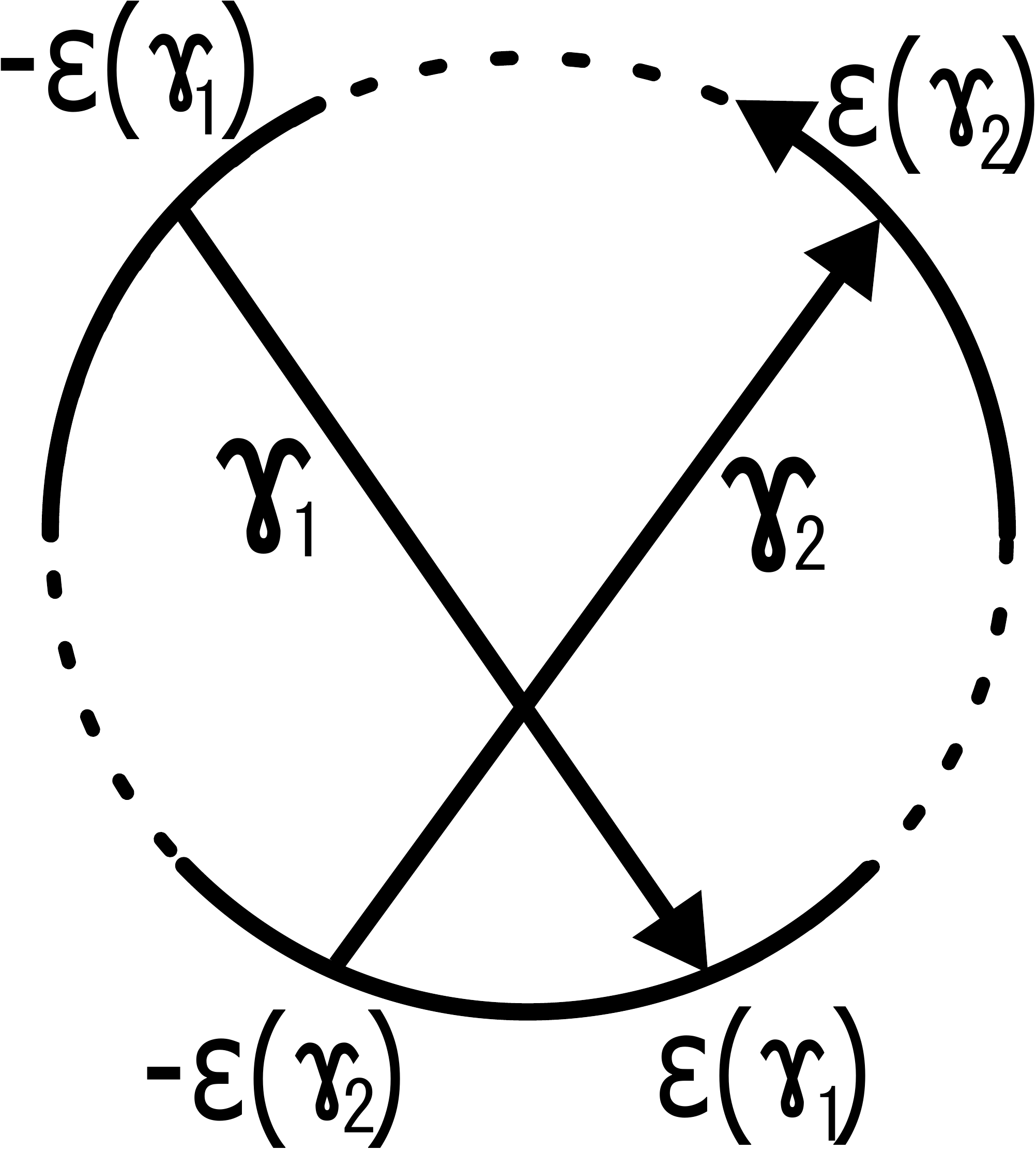} }
\put(0.3,2.8){\includegraphics[scale=.1] {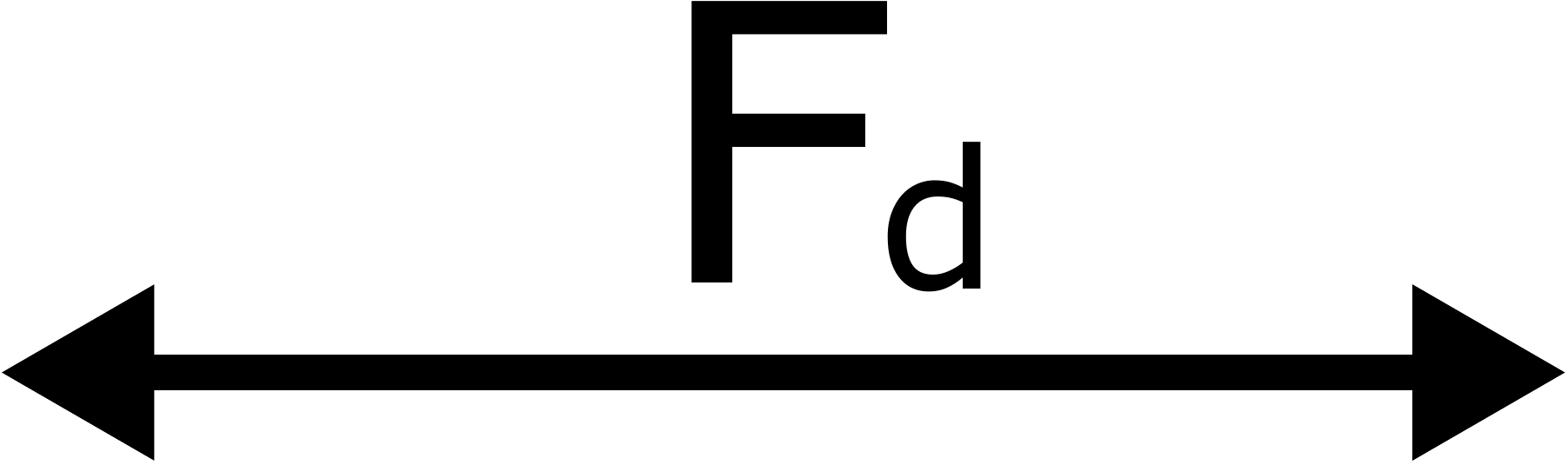} }
\put(2,1){\includegraphics[scale=.12] {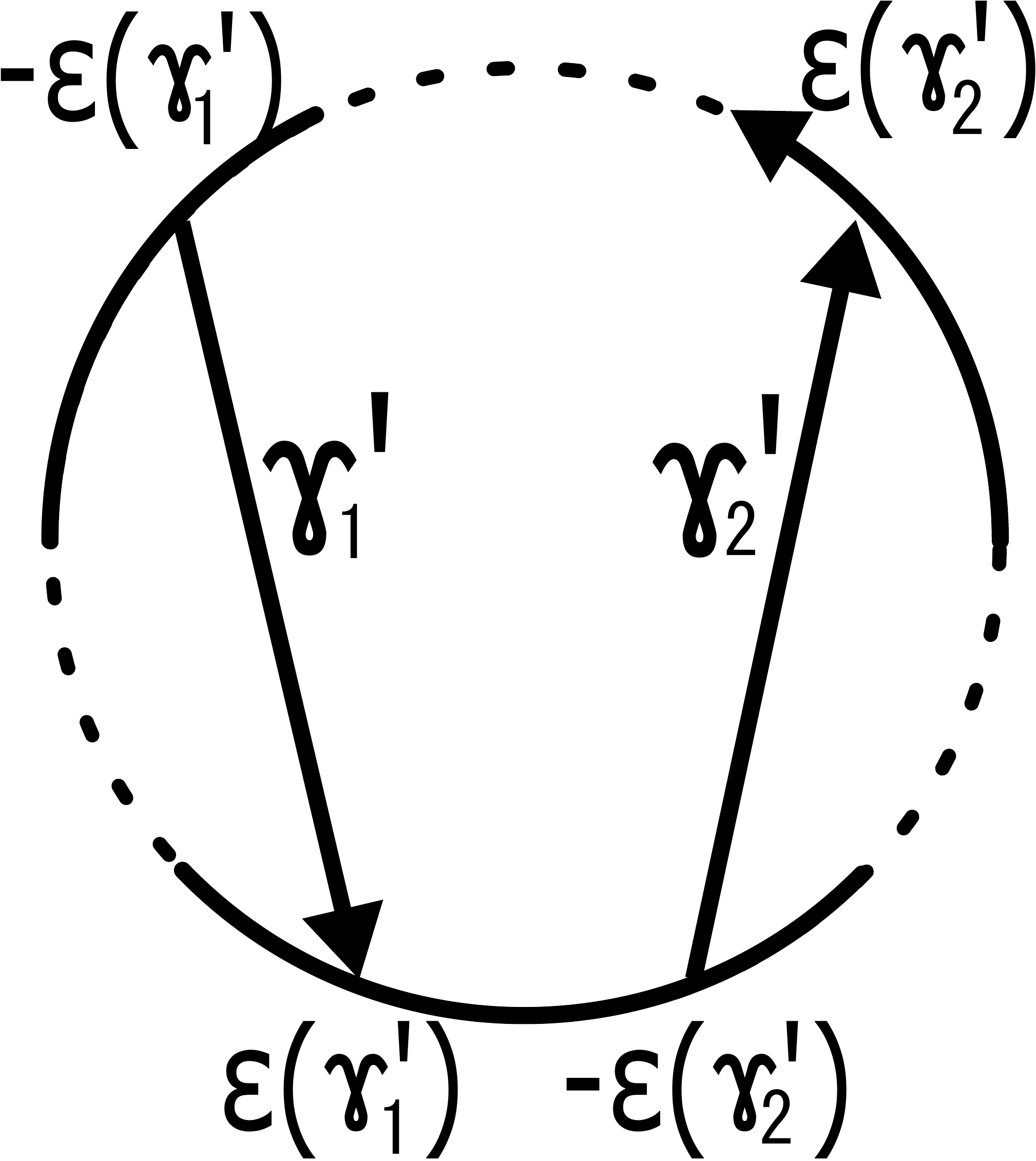} }
\put(1,0.5){(I)}
\put(6,1){\includegraphics[scale=.12] {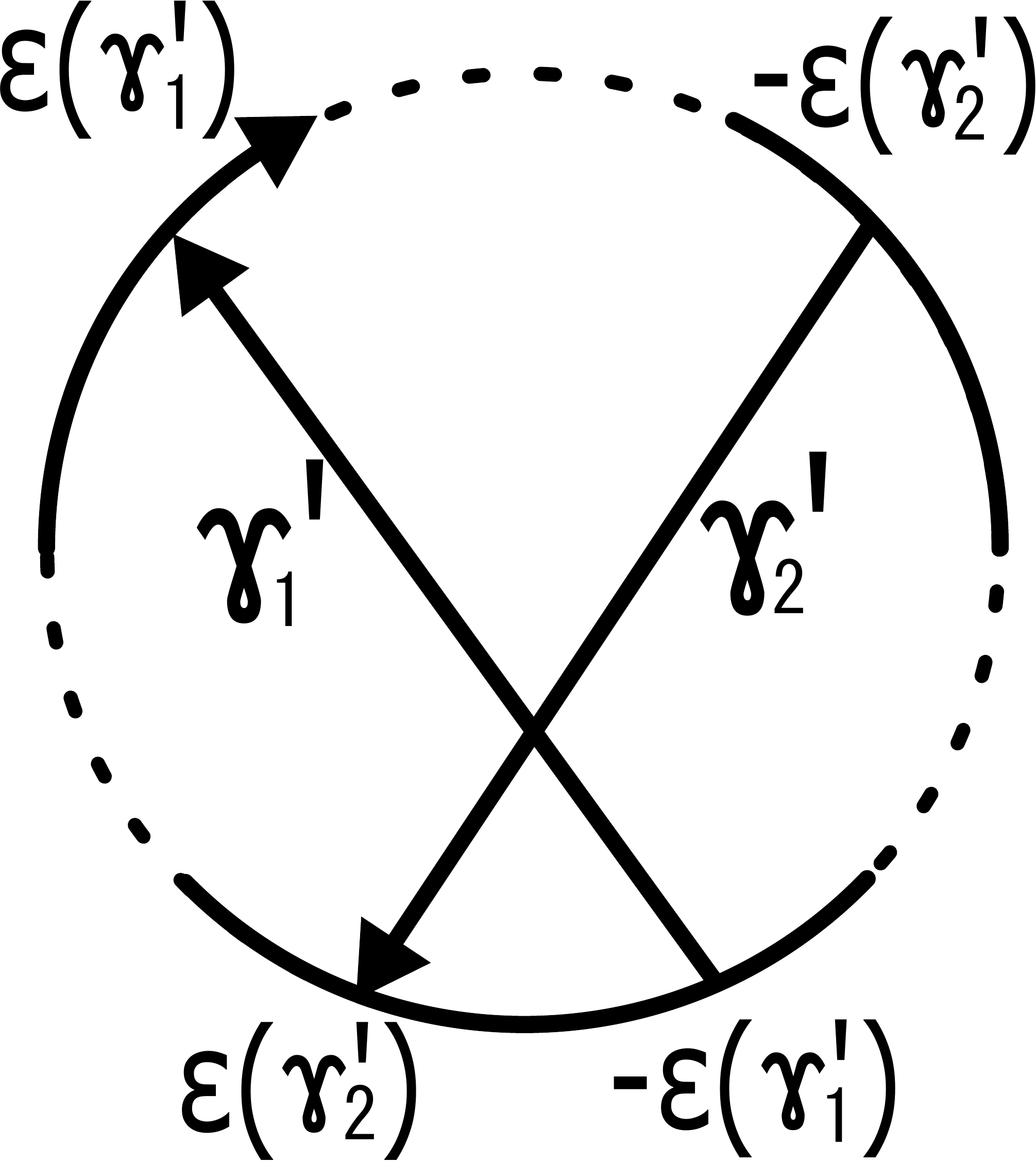} }
\put(8.3,2.8){\includegraphics[scale=.1] {fdwarrow.pdf} }
\put(10,1){\includegraphics[scale=.12] {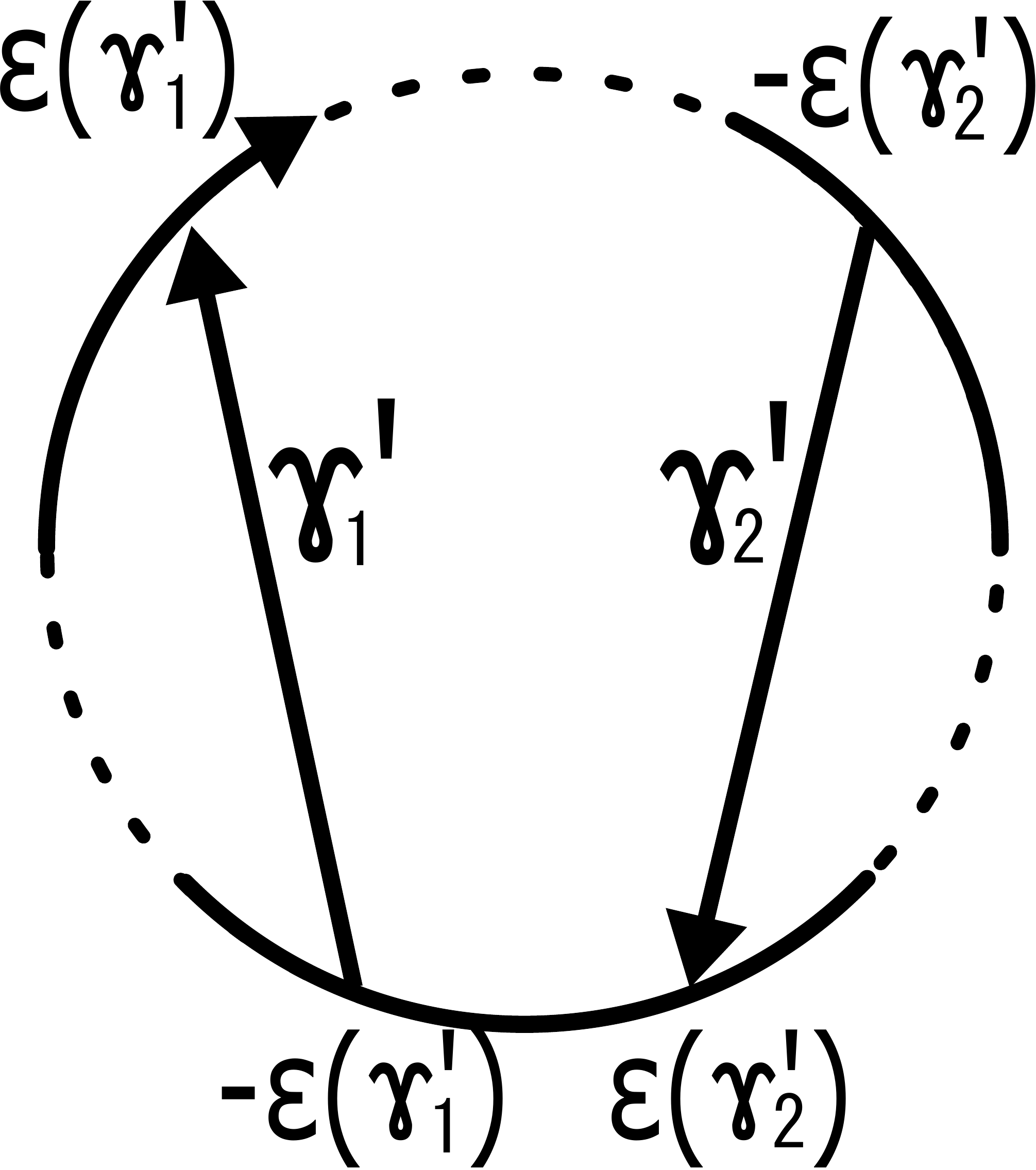} }
\put(9,0.5){(II)}
\end{picture}}
  \caption{}
  \label{ind}
\end{figure}
Therefore, we have the following.
\begin{align*}
P_{K}-P_{K'} &= \varepsilon(\gamma_{1})(t^{-i(\gamma_{1})}-1) + \varepsilon(\gamma_{2})(t^{-i(\gamma_{2})}-1)\\
&\quad -\varepsilon(\gamma_{1}')(t^{-i(\gamma_{1}')}-1) -\varepsilon(\gamma_{2}')(t^{-i(\gamma_{2}')}-1)\\
&= \varepsilon(\gamma_{1})t^{-i(\gamma_{1})} + \varepsilon(\gamma_{1})t^{-i(\gamma_{2})} -\varepsilon(\gamma_{1}')t^{-i(\gamma_{1}')} -\varepsilon(\gamma_{2}')t^{-i(\gamma_{2}')}\\
&\quad -\varepsilon(\gamma_{1}) -\varepsilon(\gamma_{2}) +\varepsilon(\gamma_{1}') +\varepsilon(\gamma_{2}')\\
&= \varepsilon(\gamma_{1})t^{-i(\gamma_{1})} + \varepsilon(\gamma_{1})t^{-i(\gamma_{2})} -\varepsilon(\gamma_{1})t^{-i(\gamma_{1})-\varepsilon(\gamma_2)} -\varepsilon(\gamma_{2})t^{-i(\gamma_{2})+\varepsilon(\gamma_{1})}\\
&=\varepsilon(\gamma_{1})t^{-i(\gamma_{1})}(1-t^{-\varepsilon(\gamma_{2})}) + \varepsilon(\gamma_{2})t^{-i(\gamma_{2})}(1-t^{\varepsilon(\gamma_{1})})\\
&= \left \{
\begin{array}{lll}
(t-1)(t^{-i(\gamma_{1})-1}-t^{-i(\gamma_{2})}) \quad &(\varepsilon(\gamma_{i}) = 1)\\
(t-1)(-t^{-i(\gamma_{1})}+t^{-i(\gamma_{2})})\quad &(\varepsilon(\gamma_{1}) = 1, \varepsilon(\gamma_{2}) = -1)\\
(t-1)(-t^{-i(\gamma_{1})-1}+t^{-i(\gamma_{2})-1})\quad &(\varepsilon(\gamma_{1}) = -1, \varepsilon(\gamma_{2}) = 1)\\
(t-1)(t^{-i(\gamma_{1})}-t^{-i(\gamma_{2})-1})\quad &(\varepsilon(\gamma_{i}) = -1)
\end{array}
\right.
\end{align*}
\end{proof}

\begin{proof}[Proof of Theorem~\ref{LBFd}]
Let $K$ be a virtual knot with a virtual knot diagram $D$ which can be transformed into the trivial knot diagram $O$ by using forbidden detour moves $s$ times. 
That is, we suppose that there exists a sequence of virtual knot diagrams $D_0, D_2, \cdots, D_s$ such that $D=D_0$, $D_i$ is obtained from $D_{i-1}$ by single forbidden detour move ($1 \le i \le s$), and $D_s = O$. 
We denote by $K_i$ the virtual knot represented by $D_i$ ($1 \le i \le s$). 
Suppose that the affine index polynomial of $K_{i}$ is expressed as $P_{t}(K_{i}) = (t-1)\sum a^{i}_{n}t^{n}$.
By Theorem~\ref{AIP}, we get the following.
\begin{align*}
P_{t}(K_{1})&=P_{t}(K) + (t-1)(\pm t^{\ell} \mp t^{m}) \\
&= (t-1)\sum a^{0}_{n}t^{n}  + (t-1)(\pm t^{\ell} \mp t^{m}) \\
&= (t-1)\cdot(\cdots + (a^{0}_{\ell} \pm 1)t^{\ell} + \cdots +(a^{0}_{m} \mp 1)t^{m} + \cdots)
\end{align*}
The coefficients of $a_{\ell}^{1}$ and $a_{\ell}^{0}$ satisfy the next ($0 \leq \ell \leq s$).
\begin{align*}
a_{\ell}^{1} &=a_{\ell}^{0} \pm1 \\
|a_{\ell}^{1}|&=|a_{\ell}^{0} \pm1| \geq |a_{\ell}^{0}| -1 
\end{align*}
Then we have the following.
\begin{align*}
\sum |a_{n}^{1}| &\geq (\sum|a_{n}^{0}|)-2 \\
\sum |a_{n}^{2}| &\geq (\sum|a_{n}^{1}|)-2 \\
&\vdots \\
\sum |a_{n}^{s}| &\geq (\sum|a_{n}^{s-1}|)-2 
\end{align*}
We conclude the following.
\begin{align*}
0 = \sum |a_{n}^{s}| &\geq (\sum |a_{n}^{0}|)-2s \\
2s &\geq \sum |a_{n}^{0}| \\ 
s &\geq \displaystyle\frac{\sum |a_{n}^{0}|}{2}
\end{align*} 
\end{proof}

\end{document}